\documentclass{article}[12pt]
\usepackage{amsmath,amsthm,amscd,amsfonts,amssymb,color}
\usepackage{cite}
\usepackage[mathscr]{eucal}
\usepackage{mathtools}

\usepackage[left=3.50cm, right=3.50cm, top=2.50cm, bottom=2.50cm]{geometry}
\usepackage[utf8]{inputenc}
\usepackage{fancyhdr}
\usepackage{lipsum}
\usepackage{semantic}
\usepackage{enumitem}
\newtheorem{theorem}{Theorem}[section]

\newtheorem{corollary}[theorem]{Corollary}

\theoremstyle{definition}
\newtheorem{definition}[theorem]{Definition}
\newtheorem{example}[theorem]{Example}
\newtheorem{Open Prob}[theorem]{Open Problem}
\theoremstyle{remark}

\begin{document}

\title{\textbf{A generalization of Kannan and Fisher fixed point theorems with an application to Volterra type integral equations}}

\author{\textbf{Rivu Bardhan$ ^{a}$}, \textbf{Cenep Özel$^b$}, \textbf{Liliana Guran$^c$} \\ \small{ $^{a}$Shiv Nadar University, Gautam Budhha Nagar, Noida, Uttar Pradesh, India}\\ \small{Email: rb212@snu.edu.in}\\ \small{$^b$ Department of Mathematics, Faculty of Sciences, King Abdulaziz University, Jeddah, Saudi Arabia }\\ \small{Email: cenep.ozel@gmail.com}\\
\small{$^c$ Department of Pharmaceutical Sciences, "Vasile Goldi\c s" Western University of Arad,
Arad, Romania }\\ \small{Email: guran.liliana@uvvg.ro}}

\normalfont
	
\date{}
\maketitle
\begin{abstract}
In this paper, we discuss about the independent types of infinite extensions to a general version of Kannan \cite{Kan} and Fisher \cite{Fis} of which the well-known Kannan and Fisher theorems come as a corollaries. We also provide a strong connection between continuous Kannan operator and Fisher operator in a restricted type of metric space. We also provide an application of the main theorem of this paper, in the field of integral equations.
\end{abstract}

\vspace*{5pt}

\noindent \textbf{Keywords:} Fixed point, Geraghty's theorem, Fisher's theorem, Kannan's theorem, complete metric space, infinite dimension.\\

\noindent \textbf{AMS Subject classification:}	47H10, 54H25.\\

\section{Introduction and Preliminaries}

\quad The fixed point theory is an important tool of nonlinear analysis very useful in the proof of the existence of a solution of different types of equations. Banach has given in 1922 the fundamental theorem of contraction operators (see \cite{Ban}) starting a new research direction. Later, Kannan and Fisher have given, totaly independent, two type of contractions, which leads to a unique fixed point for the operators. But, it is not hard to notice that the two contractions defined in \cite{Kan} and \cite{Fis} has a similar pattern.

Recently, in the paper of M. Alshumrani, H. Aydi, S. Hazra, C. Özel (see \cite{Kho}) introduced the notion of $\eta$-dimensional extension of the Geraghty's result, which give a new way to extend the Banach's fixed point result.

Let us recall further, some results given in related literature, useful in the development of our new results.

\begin{theorem} (\cite{G}) Let $(\overline{X}, d)$ be a complete metric space and $\mathcal{U} : \overline{X} \to \overline{X}$ be a
mapping. If $\mathcal{U}$ satisfies the following inequality:
\begin{align*}
d(\mathcal{U} \overline{x}, \mathcal{U} \overline{y}) \leq \beta(d(\overline{x}, \overline{y})) d(\overline{x}, \overline{y})
\end{align*}
 where $\beta : [0, \infty) \to [0, 1)$ is a function which satisfies the condition
 \begin{align*}
\lim\limits_{n\to \infty}\beta(t_n) = 1 ~implies \lim\limits_{n\to \infty}t_n = 0.
\end{align*}
\end{theorem}

\begin{theorem} (\cite{Kan}) Let $(\overline{X}, d)$ be a complete metric space and $\mathcal{U} : \overline{X} \to \overline{X}$ be a
mapping. If $\mathcal{U}$ satisfies the following inequality:
\begin{align*}
d(\mathcal{U} \overline{x}, \mathcal{U} \overline{y}) \leq c( d(\mathcal{U}\overline{x}, \overline{x})+ d(\mathcal{U}\overline{y},\overline{y}) )
\end{align*}
 where $c~\in(0,1/2)$, then $\mathcal{U}$ has a unique fixed point $\overline{w} \in X. $
\end{theorem}

\begin{theorem} (\cite{Fis}) Let $(\overline{X}, d)$ be a complete metric space and $\mathcal{U} : \overline{X} \to \overline{X}$ be a
mapping. If $\mathcal{U}$ satisfies the following inequality:
\begin{align*}
d(\mathcal{U} \overline{x}, \mathcal{U} \overline{y}) \leq c( d(\mathcal{U}\overline{x}, \overline{y})+ d(\mathcal{U}\overline{y},\overline{x}) )
\end{align*}
 where $c~\in(0,1/2)$, then $\mathcal{U}$ has a unique fixed point $\overline{w} \in \overline{X}.$
\end{theorem}


\begin{theorem}(\cite{Kho}) Let $(\overline{X}, d)$ be a complete metric space and $\eta \in \mathbb{N} $. Let $\mathcal{U} : \overline{X}^\eta \to \overline{X}$
defined such that, for all $\overline{w}_1,\overline{w}_2, . . . , \overline{w}_{\eta+1}$, we have the following inequality,
\begin{align*}
d(\mathcal{U} (\overline{w}_1, . . . , \overline{w}_\eta), \mathcal{U} (\overline{w}_2, . . . , \overline{w}_{\eta+1}))&\leq
\mathcal{M}((\overline{w}_1, . . . , \overline{w}_\eta), (\overline{w}_2, . . . , \overline{w}_{\eta+1})))\mathcal{M}((\overline{w}_1, . . . , \overline{w}_\eta), (\overline{w}_2,. . . ,  \overline{w}_{\eta+1}))
\end{align*}
where, $ \beta \in  \mathbb{G}$ and $\mathcal{M} : \overline{X}^\eta \times \overline{X}^\eta \to [0,\infty)$ is defined as
$$\mathcal{M}((\overline{w}_1, . . . , \overline{w}_\eta), (\overline{w}_2,. . . ,  \overline{w}_{\eta+1}))
= max\left\lbrace d(\overline{w}_\eta, \overline{w}_{\eta+1}), d(\overline{w}_\eta, \mathcal{U} (\overline{w}_1, . . . , \overline{w}_\eta)), d( \overline{w}_{\eta+1}, \mathcal{U} (\overline{w}_2,  . . . ,  \overline{w}_{\eta+1}))\right\rbrace.$$

Then there is a point $\overline{w} \in \overline{X}$ such that $\mathcal{U} (\overline{w}, \overline{w}, . . . , \overline{w}) = \overline{w}.$
\end{theorem}

We denote by $\mathbb{N}$ (resp. $\mathbb{N}_0$) the set of positive (nonnegative) integers.

The aim of this paper is to generalize and extend Theorem \cite{Fis} and Theorem \cite{Kan}. We denote infinite tuples of points $(\overline{x}_1, \overline{x}_2, ...)$ by $(\overline{x}_\eta)_{\i=1} ^\infty$, and the infinite tuples $(\overline{x}_1, \overline{x}_2, ...\overline{x}_{\eta-1}, \overline{x}_{\eta},\overline{x}_{\eta},...)$ with the $\eta-$th point repeated by $(\overline{x}_{\i,\widehat{\eta}})_{\i=1} ^\infty.$

Then, in our paper, we prove an infinite dimensional proper generalizations of both Kannan's Theorem \cite{Kan}, Fisher's Theorem \cite{Fis} as a corollaries of more general theorem for a family of operators each satisfying a unique family of contractions by finding its fixed points inspired by the concept of M. Al Shumrani et al. \cite{Kho}. We also prove that, there exists a homotopy between maps satisfying Kannan's contraction \cite{Kan} and Fisher's contractions \cite{Fis} respectively.

An application of our main result is also presented in the last part of the paper, to prove the existence and the uniqueness of a solution of a Volterra integral type equation.

\section{Main Results}

\quad In this section, we introduce some new different classes of generalized contraction mappings. Also, we discuss about some sufficient conditions for the existence and uniqueness of fixed points for these new classes of mappings.	Then, we generalize the notions of Gergathy function and Kannan function, which would be key tools in proving our new theorems. Moreover, to strengthen this new notions we provide some illustrative examples.

Throughout our paper, (unless otherwise stated) we assume that $\left(X,d\right)$ is a $complete$ $metric$ $space$ and $\eta\in\mathbb{N}$. Further let us introduce some new notions and we give some connected examples to validate our notions.

\begin{definition}[\textbf{Generalized Geraghty function}] A function $\beta:[0,\infty)\to[0,r)$, $r<1$ is said to be a generalized Geragthy function if and only if $\lim\limits_{n\to\infty}\beta(t_n)=r \mbox{ implies } \lim\limits_{n\to\infty}t_n=0$. Let us define the following set $\mathbb{G}_{r}:=\{\beta| \beta \text{ is a generalized Geragthy function for }r\}.$
\end{definition}

\begin{definition}[\textbf{Kannan's $\eta$ function}] For any $\mathcal{U}:\prod\limits_{l=1}^{\eta}\overline{X} \rightarrow \overline{X}$ we define the function $K_{\mathcal{U}_\eta}:\prod\limits_{l=1}^{\eta} \overline{X} \times \prod\limits_{l=1}^{\eta} \overline{X}\rightarrow [0,\infty)$, as follows
		\begin{align*}
		 K_{\mathcal{U}_\eta}((\overline{w}_{\i})_{\i=1}^{\eta},(\overline{v}_{\i})_{\i=1}^{\eta})=~d(\mathcal{U}((\overline{w}_{\i})_{\i=1}^{\eta}),\overline{w}_\eta)+d(\mathcal{U}((\overline{v}_{\i})_{\i=1}^{\eta}),\overline{v}_\eta)
		\end{align*}
	for all $(\overline{w}_{\i})_{\i=1}^{\eta},(\overline{v}_{\i})_{\i=1}^{\eta}\in \prod\limits_{\i=1}^{\eta}X$.
	\end{definition}

\begin{definition}[\textbf{extended Kannan's function}] For any $\mathcal{U}:\prod\limits_{l=1}^{\infty}\overline{X} \rightarrow \overline{X},$ we define the extended function $K^{*}_{\mathcal{U}_\eta}:\prod\limits_{l=1}^\infty \overline{X} \times \prod\limits_{l=1}^\infty \overline{X}\rightarrow [0,\infty)$, as follows
		\begin{align*}
		K^{*}_{\mathcal{U}_\eta}\left((\overline{w}_{\i,})_{\i=1}^{\infty},(\overline{v}_{\i})_{\i=1}^{\infty}\right)  &\hspace{5pt} =~(d(\mathcal{U}((\overline{w}_{\i})_{\i=1}^{\infty}),\overline{w}_\eta)+d(\mathcal{U}((\overline{v}_{\i})_{\i=1}^{\infty}),\overline{v}_\eta))
	\end{align*}
\end{definition}

\begin{definition}[\textbf{Fisher's function}] For any $\mathcal{U}:\prod\limits_{l=1}^{\eta}\overline{X} \rightarrow \overline{X},$ we define the function $F^{*}_{\eta}:\prod\limits_{l=1}^{\eta} \overline{X} \times \prod\limits_{l=1}^{\eta} \overline{X}\rightarrow [0,\infty)$,  as follows 	
		\begin{align*}
		F^{*}_{\eta}\left((\overline{w}_{\i})_{\i=1}^{\eta},(\overline{v}_{\i})_{\i=1}^{\eta}\right)  &\hspace{5pt} =~(d(\mathcal{U}((\overline{w}_{\i})_{\i=1}^{\eta}),\overline{v}_{\eta})+d(\mathcal{U}((\overline{v}_{\i,})_{\i=1}^{\eta}),\overline{w}_\eta))
	\end{align*}
\end{definition}

\begin{definition}[\textbf{extended Fisher's function}] For any $\mathcal{U}:\prod\limits_{l=1}^{\infty} \rightarrow \overline{X},$ we define the extended function $F_{\eta}:\prod\limits_{l=1}^\infty \overline{X}\times \prod\limits_{l=1}^\infty \overline{X}\rightarrow [0,\infty)$, as follows 	
		\begin{align*}
		F_{\eta}\left((\overline{w}_{\i})_{\i=1}^{\infty},(\overline{v}_{\i})_{\i=1}^{\infty}\right)  &\hspace{5pt} =~(d(\mathcal{U}((\overline{w}_{\i})_{\i=1}^{\infty}),\overline{v}_\eta)+d(\mathcal{U}((\overline{v}_{\i})_{\i=1}^{\infty}),\overline{w}_\eta))
	\end{align*}
\end{definition}

\begin{definition}[\textbf{Kannan's generalized contraction for finite dimensions}]
A function $\mathcal{U}:\prod\limits_{l=1}^{\eta}\overline{X} \rightarrow \overline{X}$, is called Kannan's generalized contraction for finite dimensions if and only if, 	
\begin{align*}
  d(\mathcal{U}((\overline{w}_{\i})_{\i=1}^{\eta}),\mathcal{U}((\overline{v}_{i})_{i=1}^{\eta}))\leq~cK^{*}_{\eta}((\overline{w}_{\i})_{\i=1}^{\eta},(\overline{v}_{\i})_{\i=1}^{\eta})),
\end{align*}
for all $\overline{w}_1,. . .,\overline{w}_{\eta+1},\overline{v}_1,. . .,\overline{v}_{\eta+1}\in X$ and some fixed $c\in(0,1/2)$.
	\end{definition}

\begin{example}\label{1}
Let $\overline{X}=[0,1]$. We define $\mathcal{U}:\prod\limits_{\i=1}^{\eta}\overline{X}\rightarrow \overline{X} $ by $$\mathcal{U}(({\overline{x}_{\i}})_{\i=1}^{\eta})=\begin{cases}
 \frac{\overline{x}_\eta}{8^\eta} , & ~\text{if } \overline{x}_\eta\in [0,1)\\
\frac{1}{16^\eta}, & ~\text{if } \overline{x}_\eta=1\\
\end{cases}$$
Then $\mathcal{U}$ is a Kannan's generalized contraction for finite dimensions for suitable $\beta$.
\end{example}

\begin{definition}[\textbf{Kannan's generalized contraction for infinite dimensions}] A function $\mathcal{U}:\prod\limits_{l=1}^{\infty}\overline{X} \rightarrow \overline{X}$ is called Kannan's generalized contraction for infinite dimensions, if and only if 	
\begin{align*}
  d(\mathcal{U}((\overline{w}_{\i,\widehat{\eta}})_{\i=1}^{\infty}),\mathcal{U}((\overline{v}_{\i,\widehat{\eta}})_{\i=1}^{\infty}))\leq~cK_{\eta}((\overline{w}_{\i,\widehat{\eta}})_{\i=1}^{\infty},(\overline{v}_{\i,\widehat{\eta}})_{\i=1}^{\infty})),
\end{align*}
for all $\overline{w}_1,. . .,\overline{w}_{\eta+1},\overline{v}_1,. . .,\overline{v}_{\eta+1}\in \overline{X}$ and some fixed $c\in(0,1/2)$.
\end{definition}

\begin{example}
In the Example \ref{1}, repeating the $\eta$-th term infinitely, we get Kannan's generalized contraction for infinite dimensions.
\end{example}

\begin{definition}[\textbf{Fisher's generalized contraction for finite dimensions}] A function $\mathcal{U}:\prod\limits_{l=1}^{\eta} \overline{X}\rightarrow \overline{X}$ is called a Fisher's generalized contraction for finite dimensions, if and only if
\begin{align*}
  d(\mathcal{U}((\overline{w}_{\i})_{\i=1}^{\eta}),\mathcal{U}((\overline{v}_{\i})_{\i=1}^{\eta}))\leq~cF^{*}_{\eta}((\overline{w}_{\i})_{\i=1}^{\eta},(\overline{v}_{\i})_{\i=1}^{\eta})),
\end{align*}
for all $\overline{w}_1,. . .,\overline{w}_{\eta+1},\overline{v}_1,. . .,\overline{v}_{\eta+1}\in \overline{X}$ and some fixed $c\in(0,1/2)$.
\end{definition}

\begin{example}\label{2}
Let $\overline{X}=[0,1]$ and we define $\mathcal{U}:\prod\limits_{\i=1}^{\eta}\overline{X}\rightarrow \overline{X} $ by $\mathcal{U}(({\overline{x}_{\i}})_{\i=1}^{\eta})= \alpha \overline{x}_\eta$ for some $\alpha \in (0,\frac{1}{2})$. Then $\mathcal{U}$ is a Fisher's generalized contraction for finite dimensions.
\end{example}

\begin{definition}[\textbf{Fisher's generalized contraction for infinite dimensions}] A function $\mathcal{U}:\prod\limits_{l=1}^{\infty}\overline{X} \rightarrow \overline{X}$ is called a Fisher's generalized contraction for infinite dimensions if and only if
\begin{align*}
  d(\mathcal{U}((\overline{w}_{\i,\widehat{\eta}})_{\i=1}^{\infty}),\mathcal{U}((\overline{v}_{\i,\widehat{\eta}})_{\i=1}^{\infty}))\leq~cF_{\eta}((\overline{w}_{\i,\widehat{\eta}})_{\i=1}^{\infty},(\overline{v}_{\i,\widehat{\eta}})_{\i=1}^{\infty})),
\end{align*}
for all $\overline{w}_1,. . .,\overline{w}_{\eta+1},\overline{v}_1,. . .,\overline{v}_{\eta+1}\in \overline{X}$ and some fixed $c\in(0,1/2)$.
\end{definition}

\begin{example}
In the Example \ref{2}, repeating the $\eta$-th term infinitely, we get Fisher's generalized contraction for infinite dimensions.
\end{example}

Consider the collection $\mathcal{F}$ of functions  $f_{\i}:\mathbb{R}_{+}^3\rightarrow \mathbb{R}_{+}; \i=1,2,3$ defined by:
\begin{itemize}
\item[$(\mathcal{F}_1)$] $f_1(\overline{x}_1,\overline{x}_2,\overline{x}_3)=c_1 \overline{x}_1+2c_2 \overline{x}_2 +2c_3 \overline{x}_3;$ where $c_1+2c_2+2c_3=1.$
\item[$(\mathcal{F}_2)$] $f_2(\overline{x}_1,\overline{x}_2,\overline{x}_3)= \max \{\overline{x}_1,\overline{x}_2,\overline{x}_3\}$.
\item[$(\mathcal{F}_3)$] $f_3(\overline{x}_1,\overline{x}_2,\overline{x}_3)=\min \{\overline{x}_1,\overline{x}_2,\overline{x}_3\}$.
\end{itemize}

\begin{definition}[\textbf{Generalized $ L_{\mathcal{U}_\eta}$ function}]\label{LC} For any $(\overline{w}_{i})_{i=1}^{\infty}, (\overline{v}_{\i})_{\i=1}^{\infty}~\in \prod\limits_{l=1}^{\infty} \overline{X}$, any operator, $\mathcal{U}:\prod\limits_{l=1}^{\infty} \overline{X}\rightarrow \overline{X}, L_{T_\eta}:[0,1]\times\prod\limits_{\i=1}^{\infty}\overline{X} \times \prod\limits_{\i=1}^{\infty}\overline{X}\rightarrow \overline{X} $ is defined as:
\begin{align*}
		L _{\mathcal{U}_\eta}\left(\alpha, (\overline{w}_{\i})_{\i=1}^{\infty},(\overline{v}_{\i})_{\i=1}^{\infty}\right)  &\hspace{5pt} =~f(B_{\eta}((\overline{w}_{\i})_{\i=1}^{\infty},(\overline{v}_{\i})_{\i=1}^{\infty}), K_{\eta}((\overline{w}_{\i})_{\i=1}^{\infty},(\overline{v}_{\i})_{\i=1}^{\infty}),F_{\eta}((\overline{w}_{\i})_{\i=1}^{\infty},(\overline{v}_{\i})_{\i=1}^{\infty}))
	\end{align*}
where $f\in \mathcal{F} $.
\end{definition}

\begin{definition}[\textbf{Generalized $L^{*}_{\mathcal{U}_\eta}$ function}] For any $(\overline{w}_{\i})_{\i=1}^{\eta}, (\overline{v}_{\i})_{\i=1}^{\eta}~\in \prod\limits_{l=1}^{\eta} \overline{X},$ and any operator, $\mathcal{U}:\prod\limits_{l=1}^{\eta} \overline{X}\rightarrow \overline{X}$, $L'^{*}_{T_\eta}:[0,1]\times\prod\limits_{\i=1}^{\eta}\overline{X} \times \prod\limits_{\i=1}^{\eta}\overline{X}\rightarrow [0,\infty) $ is defined as:
\begin{align*}
		L'^{*}_{\mathcal{U}_\eta}\left((\overline{w}_{\i})_{\i=1}^{\eta},(\overline{v}_{\i})_{\i=1}^{\eta}\right)  &\hspace{5pt} =~f(B^{*}_{\mathcal{U}_\eta}((\overline{w}_{\i})_{\i=1}^{\eta},(\overline{v}_{\i})_{\i=1}^{\eta}),K^{*}_{\mathcal{U}_\eta}((\overline{w}_{\i})_{\i=1}^{\eta},(\overline{v}_{\i})_{\i=1}^{\eta}), F^{*}_{\mathcal{U}_\eta}((\overline{w}_{\i})_{\i=1}^{\eta},(\overline{v}_{\i})_{\i=1}^{\eta}))),
 	\end{align*}
where $f\in \mathcal{F} $.
\end{definition}

\begin{definition}[\textbf{Generalized $\mathcal{M}'^{*}_{\mathcal{U}_\eta}$ function}] For any $(\overline{w}_{\i})_{\i=1}^{\eta}, (\overline{v}_{\i})_{\i=1}^{\eta}~\in \prod\limits_{l=1}^{\eta} \overline{X},$ and any operator, $\mathcal{U}:\prod\limits_{l=1}^{\eta} \overline{X}\rightarrow \overline{X}$, $\mathcal{M}'^{*}_{\mathcal{U}_\eta}:[0,1]\times\prod\limits_{i=1}^{\infty}\overline{X} \times \prod\limits_{\i=1}^{\infty}\overline{X}\rightarrow [0,\infty) $ is defined as:
\begin{align*}
		\mathcal{M}'^{*}_{\mathcal{U}_\eta}\left((\overline{w}_{\i})_{\i=1}^{\eta},(\overline{v}_{\i})_{\i=1}^{\eta}\right)  &\hspace{5pt} =~(\alpha B^{*}_{\mathcal{U}_\eta}((\overline{w}_{\i})_{\i=1}^{\eta},(\overline{v}_{\i})_{\i=1}^{\eta})+2\gamma K^{*}_{\mathcal{U}_\eta}((\overline{w}_{\i})_{\i=1}^{\eta},(\overline{v}_{\i})_{\i=1}^{\eta})+2\delta F^{*}_{\mathcal{U}_\eta}((\overline{w}_{\i})_{\i=1}^{\eta},(\overline{v}_{\i})_{\i=1}^{\eta})),
 	\end{align*}
where  $\alpha +2\gamma+ 2\delta=1$.
\end{definition}

\begin{definition}[\textbf{Generalized $\mathcal{M}_{\mathcal{U}_\eta}$ function}] For any $(\overline{w}_{\i})_{\i=1}^{\infty}, (\overline{v}_{\i})_{\i=1}^{\infty}~\in \prod\limits_{l=1}^{\infty} \overline{X}$, any operator, $\mathcal{U}:\prod\limits_{l=1}^{\infty} \overline{X}\rightarrow \overline{X}, M_{\mathcal{U}_\eta}:[0,1]\times\prod\limits_{\i=1}^{\infty}\overline{X} \times \prod\limits_{\i=1}^{\infty}\overline{X}\rightarrow \overline{X} $ is defined as:
\begin{align*}
		\mathcal{M}_{\mathcal{U}_\eta}\left(\alpha, (\overline{w}_{\i})_{\i=1}^{\infty},(\overline{v}_{\i})_{\i=1}^{\infty}\right)  &\hspace{5pt} =~\alpha K_{\eta}((\overline{w}_{\i})_{\i=1}^{\infty},(\overline{v}_{\i})_{\i=1}^{\infty})+(1-\alpha)F_{\eta}((\overline{w}_{\i})_{\i=1}^{\infty},(\overline{v}_{\i})_{\i=1}^{\infty}).
	\end{align*}
\end{definition}

\begin{definition}[\textbf{Generalized $\mathcal{M}'^{*}_{\mathcal{U}_\eta}$ function}] For any $(\overline{w}_{\i})_{\i=1}^{\eta}, (\overline{v}_{\i})_{\i=1}^{\eta}~\in \prod\limits_{l=1}^{\eta} \overline{X},$ and any operator $\mathcal{U}:\prod\limits_{l=1}^{\eta} \overline{X}\rightarrow \overline{X}$, $\mathcal{M}'^{*}_{\mathcal{U}_\eta}:[0,1]\times\prod\limits_{\i=1}^{\eta}\overline{X} \times \prod\limits_{\i=1}^{\eta}\overline{X}\rightarrow [0,\infty) $ is defined as:
\begin{align*}
		\mathcal{M}'^{*}_{\mathcal{U}_\eta}\left((\overline{w}_{\i})_{\i=1}^{\eta},(\overline{v}_{\i})_{\i=1}^{\eta}\right)  &\hspace{5pt} =~(\alpha B^{*}_{\mathcal{U}_\eta}((\overline{w}_{\i})_{\i=1}^{\eta},(\overline{v}_{\i})_{\i=1}^{\eta})+2\gamma K^{*}_{\mathcal{U}_\eta}((\overline{w}_{\i})_{\i=1}^{\eta},(\overline{v}_{\i})_{\i=1}^{\eta})+2\delta F^{*}_{\mathcal{U}_\eta}((\overline{w}_{\i})_{\i=1}^{\eta},(\overline{v}_{\i})_{\i=1}^{\eta})),
 	\end{align*}
where  $\alpha +2\gamma+ 2\delta=1$.
\end{definition}

\begin{definition}[\textbf{Generalized $\mathcal{M}'_{\mathcal{U}_\eta}$ function}] For any $(\overline{w}_{\i})_{\i=1}^{\infty}, (\overline{v}_{\i})_{\i=1}^{\infty}~\in \prod\limits_{l=1}^{\infty} \overline{X}$ , any operator $\mathcal{U}:\prod\limits_{l=1}^{\infty} \overline{X}\rightarrow \overline{X}, \mathcal{M}'_{\mathcal{U}_\eta}:[0,1]\times\prod\limits_{\i=1}^{\infty}\overline{X} \times \prod\limits_{\i=1}^{\infty}\overline{X}\rightarrow \overline{X} $ is defined as:
\begin{align*}
		\mathcal{M}'_{\mathcal{U}_\eta}\left((\overline{w}_{\i})_{\i=1}^{\infty},(\overline{v}_{\i})_{\i=1}^{\infty}\right)  &\hspace{5pt} =~(\alpha B_\eta((\overline{w}_{\i})_{\i=1}^{\infty},(\overline{v}_{\i})_{\i=1}^{\infty})+2\gamma K_{\eta}((\overline{w}_{\i})_{\i=1}^{\infty},(\overline{v}_{\i})_{\i=1}^{\infty})+2\delta F_{\eta}((\overline{w}_{\i})_{\i=1}^{\infty},(\overline{v}_{\i})_{\i=1}^{\infty})),
	\end{align*}	
where  $\alpha +2\gamma+ 2\delta=1.$	
\end{definition}

\begin{definition}\label{hc2}[\textbf{$L^{\alpha^{*}}_\eta$ contraction}] An operator $\mathcal{U}:\prod\limits_{l=1}^{\eta} \overline{X}\rightarrow \overline{X}$ is called a $L^{\alpha^{*}}_{\eta}$ $contraction$ if and only if it satisfies the following inequality,
		\begin{equation}
		 d(\mathcal{U}((\overline{w}_{\i})_{\i=1}^{\eta}),\mathcal{U}((\overline{v}_{\i})_{\i=1}^{\eta})\leq~\beta(L^{*}_{\eta}(\alpha, (w_{\i})_{\i=1}^{\eta},(\overline{v}_{\i})_{\i=1}^{\eta}))L^{*}_{\eta}(\alpha, (\overline{w}_{\i})_{\i=1}^{\eta},(\overline{v}_{\i})_{\i=1}^{\eta})
		\end{equation}
for all $\overline{w}_1, \overline{w}_2, \ldots, \overline{w}_{\eta}, \overline{v}_1, \overline{v}_2, \ldots, \overline{v}_\eta \in \overline{X}$, for a particular $\beta\in\mathbb{G}_{\frac{1}{2}}$.
\end{definition}

\begin{definition}[\textbf{Generalized $L^{\alpha}_{\eta}$ contraction}] An operator $\mathcal{U}:\prod\limits_{l=1}^{\infty} \overline{X}\rightarrow \overline{X}$ is called a $L^{\alpha}_{\eta}$ $contraction$ if and only if it satisfies the following inequality,
		\begin{equation}\label{hc3}
		 d(\mathcal{U}((\overline{w}_{\i,\widehat{\eta}})_{\i=1}^{\infty}),\mathcal{U}((\overline{v}_{\i,\widehat{\eta}})_{\i=1}^{\infty})\leq~\beta(L_{\mathcal{U}_\eta}(\alpha, (\overline{w}_{\i,\widehat{\eta}})_{\i=1}^{\infty},(\overline{v}_{\i,\widehat{\eta}})_{\i=1}^{\infty}))L_{\mathcal{U}_\eta}(\alpha, (\overline{w}_{\i,\widehat{\eta}})_{\i=1}^{\infty},(\overline{v}_{\i,\widehat{\eta}})_{\i=1}^{\infty})
		\end{equation}
	for all	$\overline{w}_1, \overline{w}_2, \ldots, \overline{w}_{\eta}, \overline{v}_1, \overline{v}_2, \ldots, \overline{v}_\eta \in \overline{X}$, for a particular $\beta\in\mathbb{G}_{\frac{1}{2}}$.
\end{definition}

\begin{definition}[\textbf{$H^{\alpha^{*}}_\eta$ contraction):}] An operator $\mathcal{U}:\prod\limits_{l=1}^{\eta} \overline{X}\rightarrow \overline{X}$ is called a $H^{\alpha^{*}}_{\eta}$ $contraction$ if and only if it satisfies the following inequality,
		\label{hc}
		\begin{equation}
		 d(\mathcal{U}((\overline{w}_{\i})_{\i=1}^{\eta}),\mathcal{U}((\overline{v}_{\i})_{\i=1}^{\eta})\leq~\beta(\mathcal{M}^{*}_{\eta}(\alpha, (\overline{w}_{\i})_{\i=1}^{\eta},(\overline{v}_{\i})_{\i=1}^{\eta}))\mathcal{M}^{*}_{\eta}(\alpha, (\overline{w}_{\i})_{i=1}^{\eta},(\overline{v}_{\i})_{\i=1}^{\eta})
		\end{equation}
		for all $ \overline{w}_1, \overline{w}_2, \ldots, \overline{w}_{\eta}, \overline{v}_1, \overline{v}_2, \ldots, \overline{v}_\eta \in \overline{X}$, for a particular $\beta\in\mathbb{G}_{\frac{1}{2}}$.
\end{definition}

\begin{example}
Let $\overline{X}=[0,1]$ and  define $\mathcal{U}:\prod\limits_{\i=1}^{\eta}\overline{X}\rightarrow \overline{X} $ by $\mathcal{U}(({\overline{x}_{\i}})_{\i=1}^{\eta})= \alpha \overline{x}_\eta$ for some $\alpha \in (0,\frac{1}{2})$. Then $\mathcal{U}$ is a $H_K^{\alpha}$ contraction.
\end{example}

\begin{definition}[\textbf{Generalized $H^{\alpha}_\eta$ contraction}] An operator $\mathcal{U}:\prod\limits_{l=1}^{\infty} \overline{X}\rightarrow \overline{X}$ is called a generalized $H^{\alpha}_{\eta}$ $contraction$ if and only if it satisfies the following inequality,
		\begin{equation}\label{hc1}
		 d(\mathcal{U}((\overline{w}_{\i,\widehat{\eta}})_{\i=1}^{\infty}),\mathcal{U}((\overline{v}_{\i,\widehat{\eta}})_{\i=1}^{\infty})\leq~\beta(\mathcal{M}_{\mathcal{U}_\eta}(\alpha, (\overline{w}_{\i,\widehat{\eta}})_{\i=1}^{\infty},(\overline{v}_{\i,\widehat{\eta}})_{\i=1}^{\infty}))\mathcal{M}_{\mathcal{U}_\eta}(\alpha, (\overline{w}_{\i,\widehat{\eta}})_{\i=1}^{\infty},(\overline{v}_{\i,\widehat{\eta}})_{\i=1}^{\infty})
		\end{equation}
for all	$ \overline{w}_1, \overline{w}_2, \ldots, \overline{w}_{\eta}, \overline{v}_1, \overline{v}_2, \ldots, \overline{v}_\eta \in \overline{X}$, for a particular $\beta\in\mathbb{G}_{\frac{1}{2}}$.
\end{definition}

\begin{example}
Let $\overline{X}=[0,1]$ and  define $\mathcal{U}:\prod\limits_{\i=1}^{\infty}\overline{X}\rightarrow \overline{X} $ by $$\mathcal{U}(({\overline{x}_{\i}})_{\i=1}^{\infty})=\begin{cases}
 \frac{\overline{x}_\eta}{20} , & ~\text{if }\overline{ x}_\eta\in [0,1)\\
\frac{1}{18}  , & ~\text{if } \overline{x}_\eta=1.
\end{cases}$$  Take, $\beta=\frac{1}{4}$ and $\alpha=\frac{1}{4}.$ Then $\mathcal{U}$ is a generalized $H_K^{\alpha}$ contraction.
\end{example}

\begin{definition}\label{hc}[\textbf{$\mathcal{C}_\eta$ contraction}] An operator $\mathcal{U}:\prod\limits_{l=1}^{\eta} \overline{X}\rightarrow \overline{X}$ is called a $\mathcal{C}_{\eta}$ $contraction$ if and only if it satisfies the following inequality,
		\begin{equation}	d(\mathcal{U}((\overline{w}_{\i})_{\i=1}^{\eta}),\mathcal{U}((\overline{v}_{\i})_{\i=1}^{\eta})\leq~\beta(\mathcal{M}'^{*}_{\eta}((\overline{w}_{\i})_{\i=1}^{\eta},(\overline{v}_{\i})_{\i=1}^{\eta}))\mathcal{M}'^{*}_{\eta}((\overline{w}_{\i})_{\i=1}^{\eta},(\overline{v}_{\i})_{\i=1}^{\eta})
		\end{equation}
for all $\overline{w}_1, \overline{w}_2, \ldots, \overline{w}_{\eta}, \overline{v}_1, \overline{v}_2, \ldots, \overline{v}_\eta \in \overline{X}$, for a particular $\beta\in\mathbb{G}_{\frac{1}{2}}$.
		\end{definition}

\begin{example}\label{6}
Let $\overline{X}=[0,1]$ and we define $\mathcal{U}:\prod\limits_{\i=1}^{\infty}\overline{X}\rightarrow \overline{X} $ by $\mathcal{U}((\overline{x}_{\i})_{\i=1}^{\eta})= \alpha \overline{x}_\eta$ for some $\alpha \in (0,\frac{1}{2})$. Then $\mathcal{U}$ is a $\mathcal{C}_\eta$ contraction.
\end{example}

\begin{definition}[\textbf{Generalized $\mathcal{C}_\eta$ contraction}] An operator $\mathcal{U}:\prod\limits_{l=1}^{\infty} \overline{X}\rightarrow \overline{X}$ is called a generalized $\mathcal{C}_{\eta}$ $contraction$ if and only if it satisfies the following inequality,
		\begin{equation}\label{hc1}		 d(\mathcal{U}((\overline{w}_{\i,\widehat{\eta}})_{\i=1}^{\infty}),\mathcal{U}((\overline{v}_{\i,\widehat{\eta}})_{\i=1}^{\infty})\leq~\beta(\mathcal{M}'_{\mathcal{U}_\eta}((\overline{w}_{\i,\widehat{\eta}})_{\i=1}^{\infty},(\overline{v}_{\i,\widehat{\eta}})_{\i=1}^{\infty}))\mathcal{M}'_{\mathcal{U}_\eta}((\overline{w}_{\i,\widehat{\eta}})_{\i=1}^{\infty},(\overline{v}_{\i,\widehat{\eta}})_{\i=1}^{\infty})
		\end{equation}
for all $\overline{w}_1, \overline{w}_2, \ldots, \overline{w}_{\eta}, \overline{v}_1, \overline{v}_2, \ldots, \overline{v}_\eta \in \overline{X}$, for a particular $\beta\in\mathbb{G}_{\frac{1}{2}}$.
		\end{definition}

\begin{example}
In the Example \ref{6}, repeating the $\eta$-th term infinitely we get generalized $\mathcal{C}_\eta$ contraction.
\end{example}

\par Next, let us give the first main theorem of our work, which is a generalization of the results given in \cite{Kan} and \cite{Fis}.

\begin{theorem}\label{3}
Let $\left(\overline{X},d\right)$ be a complete metric space and let $\mathcal{U}:\prod\limits_{l=1}^\infty \overline{X} \rightarrow \overline{X}$ be a generalized $\mathcal{C}_{\eta}$ $contraction$ for some $ \eta\in\mathbb{N} .$ Then the following holds: \begin{itemize}
\item[$(1)$] $\mathcal{U}$ has a unique fixed point, i.e. there exists $\overline{u}\in \overline{X}$ such that $\mathcal{U}((\overline{w})_{\i=1}^{\infty})=\overline{w}.$
\item[$(2)$] For any $\overline{x}_1, \overline{x}_2,...,\overline{x}_\eta\in X,$ the infinite $\eta-$Picard sequence converges to $\overline{w}.$
\end{itemize}
\end{theorem}

\begin{proof}
Let $\overline{x}_1, \overline{x}_2, ...,\overline{x}_\eta \in X,$ be arbitrary but fixed. We construct an infinite $\eta-$Picard sequence as:$$\overline{x}_{n+\eta}=\mathcal{U}\left(\left(\overline{x}_{n+\i-1,\widehat{n+\eta-1}}\right)_{\i=1}^{\infty}\right); \text{ for all } n\in \mathbb{N}. $$

We assert that $\lim_{\eta\rightarrow \infty} d(\overline{x}_{n+\eta+1},\overline{x}_{n+\eta+2})=0.$

Then, we have, \begin{align*}
& d\left(\overline{x}_{n+\eta+1},\overline{x}_{n+\eta+2}\right)\\
& = d\left(\mathcal{U}\left(\left(\overline{x}_{n+\i,\widehat{ n+\eta}}\right)_{\i=1}^{\infty} \right),\mathcal{U}\left(\left(\overline{x}_{n+\i+1,\widehat{ n+\eta+1}} \right)_{\i=1}^{\infty}\right)\right)\\
& \leq \beta \left(\mathcal{M}'_{\mathcal{U}_\eta} \left(\left(\overline{x}_{n+\i,\widehat{ n+\eta}}\right)_{\i=1}^{\infty},\left(\overline{x}_{n+\i+1,\widehat{n+\eta+1}} \right)_{\i=1}^{\infty}\right)\right) \mathcal{M}'_{\mathcal{U}_\eta} \left(\left(\overline{x}_{n+\i,\widehat{ n+\eta}}\right)_{\i=1}^{\infty},\left(\overline{x}_{n+\i+1,\widehat{n+\eta+1}} \right)_{\i=1}^{\infty}\right).
\end{align*}

Also, we have,\begin{align*}
& \mathcal{M}'_{\mathcal{U}_\eta} \left(\left(\overline{x}_{n+\i,\widehat{ n+\eta}}\right)_{\i=1}^{\infty},\left(\overline{x}_{n+\i+1,\widehat{n+\eta+1}} \right)_{\i=1}^{\infty}\right)\\
& = \alpha B_\eta \left(\left(\overline{x}_{n+\i,\widehat{ n+\eta}}\right)_{\i=1}^{\infty},\left(\overline{x}_{n+\i+1,\widehat{n+\eta+1}} \right)_{\i=1}^{\infty}\right)+2\gamma K_\eta \left( \left(\overline{x}_{n+\i,\widehat{ n+\eta}}\right)_{i=1}^{\infty},\left(\overline{x}_{n+\i+1,\widehat{n+\eta+1}} \right)_{\i=1}^{\infty}\right)\\
& + 2\delta F_\eta \left( \left(\overline{x}_{n+\i,\widehat{ n+\eta}}\right)_{\i=1}^{\infty},\left(\overline{x}_{n+\i+1,\widehat{n+\eta+1}} \right)_{\i=1}^{\infty}\right).
\end{align*}

On the other hand, \begin{align*}
& B_\eta \left( \left(\overline{x}_{n+\i,\widehat{ n+\eta}}\right)_{\i=1}^{\infty},\left(\overline{x}_{n+\i+1,\widehat{n+\eta+1}} \right)_{\i=1}^{\infty}\right)\\
& = d\left(\overline{x}_{n+\eta},\overline{x}_{n+\eta+1}\right)
\end{align*}

and, \begin{align*}
& K_\eta \left( \left(\overline{x}_{n+\i,\widehat{ n+\eta}}\right)_{\i=1}^{\infty},\left(\overline{x}_{n+\i+1,\widehat{n+\eta+1}} \right)_{\i=1}^{\infty}\right)\\
& = d\left( \mathcal{U}\left(\left(\overline{x}_{n+\i,\widehat{ n+\eta}}\right)_{\i=1}^{\infty}\right), \overline{x}_{n+\eta}\right)+d\left(\mathcal{U}\left(\left(\overline{x}_{n+\i+1,\widehat{n+\eta+1}} \right)_{\i=1}^{\infty}\right), \overline{x}_{n+\eta+1}\right)\\
& = d\left(\overline{x}_{n+\eta+1},\overline{x}_{n+\eta}\right)+d\left(\overline{x}_{n+\eta+2},\overline{x}_{n+\eta+1}\right).
\end{align*}

Also, \begin{align*}
& F_\eta \left( \left(\overline{x}_{n+\i,\widehat{ n+\eta}}\right)_{\i=1}^{\infty},\left(\overline{x}_{n+\i+1,\widehat{n+\eta+1}} \right)_{\i=1}^{\infty}\right)\\
& = d\left( \mathcal{U}\left(\left(\overline{x}_{n+\i,\widehat{ n+\eta}}\right)_{\i=1}^{\infty}\right), \overline{x}_{n+\eta+1}\right)+d\left(\mathcal{U}\left(\left(\overline{x}_{n+\i+1,\widehat{n+\eta+1}} \right)_{\i=1}^{\infty}\right), \overline{x}_{n+\eta}\right)\\
& = d\left(\overline{x}_{n+\eta+1},\overline{x}_{n+\eta+1}\right)+d\left(\overline{x}_{n+\eta+2},\overline{x}_{n+\eta}\right)\\
& \leq  d\left(\overline{x}_{n+\eta+1},\overline{x}_{n+\eta+2}\right)+d\left(\overline{x}_{n+\eta+1},\overline{x}_{n+\eta}\right).
\end{align*}

Thus, we get
\begin{align}
& d\left(\overline{x}_{n+\eta+1},\overline{x}_{n+\eta+2}\right)  \nonumber \\
& \leq \beta \left(\mathcal{M}'_{\mathcal{U}_\eta} \left(\left(\overline{x}_{n+i,\widehat{ n+\eta}}\right)_{\i=1}^{\infty},\left(\overline{x}_{n+\i+1,\widehat{n+\eta+1}} \right)_{\i=1}^{\infty}\right)\right) \mathcal{M}'_{\mathcal{U}_\eta} \left(\left(\overline{x}_{n+\i,\widehat{ n+\eta}}\right)_{\i=1}^{\infty},\left(\overline{x}_{n+\i+1,\widehat{n+\eta+1}} \right)_{\i=1}^{\infty}\right)  \nonumber \\
& < \frac{1}{2}\left[\alpha \{d\left(\overline{x}_{n+\eta},\overline{x}_{n+\eta+1}\right)+2\gamma \{d\left(\overline{x}_{n+\eta+1},\overline{x}_{n+\eta}\right)+ d\left(\overline{x}_{n+\eta+2},\overline{x}_{n+\eta+1}\right)\}\right.\nonumber\\
& \left.~~~~~ +2\delta \{d\left(\overline{x}_{n+\eta+1},\overline{x}_{n+\eta+2}\right)+d\left(\overline{x}_{n+\eta+1},\overline{x}_{n+\eta}\right)\}\right]. \tag{1}\label{eq_tag1}
\end{align}

Since, $\alpha +2\gamma+ 2\delta=1,$ we have, $$d\left(\overline{x}_{n+\eta+1},\overline{x}_{n+\eta+2}\right) < \frac{1}{2} [d\left(\overline{x}_{n+\eta},\overline{x}_{n+\eta+1}\right)+ d\left(\overline{x}_{n+\eta+1},\overline{x}_{n+\eta+2}\right)].$$

Therefore, $$d\left(\overline{x}_{n+\eta+1},\overline{x}_{n+\eta+2}\right)<d\left(\overline{x}_{n+\eta},\overline{x}_{n+\eta+1}\right).$$ This shows that $\{d\left(\overline{x}_{n+\eta},\overline{x}_{n+\eta+1}\right)\}$ is a decreasing sequence of non-negative real numbers.

Hence, there exists some $L\geq 0$ such that, $$\lim_{n\rightarrow \infty}  d\left(\overline{x}_{n+\eta},\overline{x}_{n+\eta+1}\right)=L.$$
We claim that $L=0$. If not, then let $L>0.$ \\ So, $$\lim_{n\rightarrow  \infty} B_\eta \left( \left(\overline{x}_{n+\i,\widehat{ n+\eta}}\right)_{\i=1}^{\infty},\left(\overline{x}_{n+\i+1,\widehat{n+\eta+1}} \right)_{\i=1}^{\infty}\right)=L;$$
$$\lim_{n\rightarrow  \infty} K_\eta \left( \left(\overline{x}_{n+i,\widehat{ n+\eta}}\right)_{\i=1}^{\infty},\left(\overline{x}_{n+\i+1,\widehat{n+\eta+1}} \right)_{\i=1}^{\infty}\right)=2L;$$
and $$\lim_{n\rightarrow  \infty} F_\eta \left( \left(\overline{x}_{n+\i,\widehat{ n+\eta}}\right)_{\i=1}^{\infty},\left(\overline{x}_{n+i+1,\widehat{n+\eta+1}} \right)_{i=1}^{\infty}\right)\leq2L.$$
Thus, $$\lim_{n\rightarrow  \infty} \mathcal{M}'_{\mathcal{U}_\eta} \left( \left(\overline{x}_{n+i,\widehat{ n+\eta}}\right)_{\i=1}^{\infty},\left(\overline{x}_{n+\i+1,\widehat{n+\eta+1}} \right)_{i=1}^{\infty}\right)\leq2L.$$

Then from \eqref{eq_tag1} we get
 \begin{align*}
& L \leq \lim_{n\rightarrow  \infty}\beta \left(\mathcal{M}'_{\mathcal{U}_\eta} \left(,\left(\overline{x}_{n+\i,\widehat{ n+\eta}}\right)_{\i=1}^{\infty},\left(\overline{x}_{n+\i+1,\widehat{n+\eta+1}} \right)_{\i=1}^{\infty}\right)\right)\cdot 2L\\
 \Rightarrow & \beta \left(\mathcal{M}'_{\mathcal{U}_\eta} \left(\left(\overline{x}_{n+\i,\widehat{ n+\eta}}\right)_{\i=1}^{\infty},\left(\overline{x}_{n+\i+1,\widehat{n+\eta+1}} \right)_{\i=1}^{\infty}\right)\right)\geq \frac{1}{2}\\
 \text{But, }& \lim_{n\rightarrow  \infty}\beta \left(\mathcal{M}'_{\mathcal{U}_\eta} \left(\left(\overline{x}_{n+\i,\widehat{ n+\eta}}\right)_{\i=1}^{\infty},\left(\overline{x}_{n+\i+1,\widehat{n+\eta+1}} \right)_{\i=1}^{\infty}\right)\right)\geq \frac{1}{2}\\
 \Rightarrow & \lim_{n\rightarrow  \infty}\beta \left(\mathcal{M}'_{\mathcal{U}_\eta} \left(\alpha,\left(\overline{x}_{n+\i,\widehat{ n+\eta}}\right)_{\i=1}^{\infty},\left(\overline{x}_{n+\i+1,\widehat{n+\eta+1}} \right)_{\i=1}^{\infty}\right)\right)= \frac{1}{2}\\
 \Rightarrow & \lim_{n\rightarrow  \infty} \mathcal{M}'_{\mathcal{U}_\eta} \left(\alpha,\left(\overline{x}_{n+\i,\widehat{ n+\eta}}\right)_{\i=1}^{\infty},\left(\overline{x}_{n+\i+1,\widehat{n+\eta+1}} \right)_{\i=1}^{\infty}\right)= 0\\
 \Rightarrow &  \lim_{n\rightarrow  \infty}\max \{d\left(\overline{x}_{n+\eta},\overline{x}_{n+\eta+1}\right),d\left(\overline{x}_{n+\eta+2},\overline{x}_{n+\eta}\right)\}=0\\
 \Rightarrow & \lim_{n\rightarrow  \infty} d\left(\overline{x}_{n+\eta},\overline{x}_{n+\eta+1}\right)=0; \text{i.e. } L=0.
\end{align*}

Contradiction. Therefore,
\begin{align}
\lim_{n\rightarrow  \infty} d\left(\overline{x}_{n+\eta},\overline{x}_{n+\eta+1}\right)=0.  \tag{2}\label{eq_tag2}
\end{align}

Next, we prove that the sequence $\{\overline{x}_{n+\eta}\}_{n\in \mathbb{N}}$ is a Cauchy sequence. We prove this by contradiction. Then let us suppose that $\{(\overline{x}_{n+\eta})_{n\in \mathbb{N}}\}$ is not a Cauchy sequence. Then there exists $\epsilon_{0}>0$ such that, for all $\eta\in \mathbb{N}$, there exists $p\leq n_p\leq m_p$ such that, $d\left(\overline{x}_{n(p)+\eta},\overline{x}_{m(p)+\eta}\right)\geq \epsilon_{0}$; $ \text{where } \{\overline{x}_{n(p)}\}_{p\in \mathbb{N}} \text{ and } \{\overline{x}_{m(p)}\}_{p\in \mathbb{N}} \text{ are two subsequences of }\{\overline{x}_{n+\eta}\}_{n\in \mathbb{N}}.$

 Then, for each $n(p)$ one can choose least of such $m_{p}$ so that, $$d\left(\overline{x}_{n(p)+\eta},\overline{x}_{m(p)+\eta-1}\right)< \epsilon_{0}.$$

Then \begin{align*}
& d(\overline{x}_{n(p)+\eta-1},\overline{x}_{m(p)+\eta-1}) \\
& \leq d(\overline{x}_{n(p)+\eta},\overline{x}_{n(p)+\eta-1})+ d(\overline{x}_{n(p)+\eta},\overline{x}_{m(p)+\eta-1}). \\
\text{i.e., } & \lim_{p\rightarrow \infty} d(\overline{x}_{n(p)+\eta-1},\overline{x}_{m(p)+\eta-1}) \leq \epsilon_{0}. \tag{3}\label{eq_tag3} \\
\text{Also, }  \epsilon_{0} & \leq d\left(\overline{x}_{n(p)+\eta},\overline{x}_{m(p)+\eta}\right)\\
& \leq d\left(\overline{x}_{n(p)+\eta},\overline{x}_{n(p)+\eta-1}\right)+ d\left(\overline{x}_{n(p)+\eta-1},\overline{x}_{m(p)+\eta-1}\right)+d(\overline{x}_{m(p)+\eta-1},\overline{x}_{m(p)+\eta}).\\
\text{i.e., } & \lim_{p\rightarrow \infty} d\left(\overline{x}_{n(p)+\eta-1},\overline{x}_{m(p)+\eta-1}\right) \geq \epsilon_{0}. \tag{4}\label{eq_tag4}
\end{align*}

Therefore, by \eqref{eq_tag3} and \eqref{eq_tag4} we get
$$\lim_{p\rightarrow \infty} d\left(\overline{x}_{n(p)+\eta-1},\overline{x}_{m(p)+\eta-1}\right) =\epsilon_{0} .$$

Next, follow \begin{align*}& B_\eta \left( \left(\overline{x}_{n(p)+\i-1,\widehat{ n(p)+\eta-1}}\right)_{\i=1}^{\infty},\left(\overline{x}_{m(p)+\i-1,\widehat{m(p)+\eta-1}} \right)_{\i=1}^{\infty}\right)\\
& = d\left(\overline{x}_{n(p)+\eta-1},\overline{x}_{m(p)+\eta-1}\right).\\
\end{align*}

Therefore, $$\lim_{p\rightarrow \infty}  B_\eta \left( \left(\overline{x}_{n(p)+\i-1,\widehat{ n(p)+\eta-1}}\right)_{\i=1}^{\infty},\left(\overline{x}_{m(p)+\i-1,\widehat{m(p)+\eta-1}} \right)_{\i=1}^{\infty}\right)= \epsilon_0.$$
\begin{align*}
& K_\eta \left( \left(\overline{x}_{n(p)+\i-1,\widehat{ n(p)+\eta-1}}\right)_{\i=1}^{\infty},\left(\overline{x}_{m(p)+\i-1,\widehat{m(p)+\eta-1}} \right)_{\i=1}^{\infty}\right)\\
& = d\left( \mathcal{U}\left(\left(\overline{x}_{n(p)+\i-1,\widehat{ n(p)+\eta-1}}\right)_{\i=1}^{\infty}\right), \overline{x}_{n(p)+\eta-1}\right)+ d\left(\mathcal{U}\left(\left(\overline{x}_{m(p)+\i-1,\widehat{m(p)+\eta-1}} \right)_{\i=1}^{\infty}\right), \overline{x}_{m(p)+\eta+1}\right)\\
& = d\left(\overline{x}_{n(p)+\eta},\overline{x}_{n(p)+\eta-1}\right)+d\left(\overline{x}_{m(p)+\eta},\overline{x}_{m(p)+\eta-1}\right).\\
\text{Thus, }& \lim_{p\rightarrow \infty}  K_\eta \left( \left(\overline{x}_{n(p)+\i-1,\widehat{ n(p)+\eta-1}}\right)_{\i=1}^{\infty},\left(\overline{x}_{m(p)+\i-1,\widehat{m(p)+\eta-1}} \right)_{\i=1}^{\infty}\right)=0.
\end{align*}

Again, \begin{align*}
& F_\eta \left( \left(\overline{x}_{n(p)+\i-1,\widehat{ n(p)+\eta-1}}\right)_{\i=1}^{\infty},\left(\overline{x}_{m(p)+\i-1,\widehat{m(p)+\eta-1}} \right)_{\i=1}^{\infty}\right)\\
& = d\left( \mathcal{U}\left(\left(\overline{x}_{n(p)+\i-1,\widehat{ n(p)+\eta-1}}\right)_{\i=1}^{\infty}\right), \overline{x}_{m(p)+\eta-1}\right)+d\left(\mathcal{U}\left(\left(\overline{x}_{m(p)+\i-1,\widehat{m(p)+\eta-1}} \right)_{\i=1}^{\infty}\right), \overline{x}_{n(p)+\eta-1}\right)\\
& = d\left(\overline{x}_{n(p)+\eta},\overline{x}_{m(p)+\eta-1}\right)+d\left(\overline{x}_{m(p)+\eta},x_{n(p)+\eta-1}\right).\\
\text{Thus, }& \lim_{p\rightarrow \infty} F_\eta \left( \left(\overline{x}_{n(p)+\i-1,\widehat{ n(p)+\eta-1}}\right)_{\i=1}^{\infty},\left(\overline{x}_{m(p)+\i-1,\widehat{m(p)+\eta-1}} \right)_{\i=1}^{\infty}\right)\leq 2\epsilon_{0}.
\end{align*}

Now, \begin{align*}
& d\left(\overline{x}_{n(p)+\eta},\overline{x}_{m(p)+\eta}\right)+d\left(\overline{x}_{n(p)+\eta-1},\overline{x}_{m(p)+\eta-1}\right)\\
& \leq d\left(\overline{x}_{n(p)+\eta},\overline{x}_{m(p)+\eta-1}\right)+d\left(\overline{x}_{m(p)+\eta-1},\overline{x}_{m(p)+\eta}\right)+d\left(\overline{x}_{n(p)+\eta-1},\overline{x}_{m(p)+\eta}\right)+d\left(\overline{x}_{m(p)+\eta},\overline{x}_{m(p)+\eta-1}\right).
\end{align*}
i.e., $$\lim_{p\rightarrow \infty} F_\eta \left( \left(\overline{x}_{n(p)+\i-1,\widehat{ n(p)+\eta-1}}\right)_{\i=1}^{\infty},\left(\overline{x}_{m(p)+\i-1,\widehat{m(p)+\eta-1}} \right)_{\i=1}^{\infty}\right)\geq 2\epsilon_{0}.$$

Thus, $$\lim_{p\rightarrow \infty} F_\eta \left( \left(\overline{x}_{n(p)+\i-1,\widehat{ n(p)+\eta-1}}\right)_{\i=1}^{\infty},\left(\overline{x}_{m(p)+\i-1,\widehat{m(p)+\eta-1}} \right)_{\i=1}^{\infty}\right)= 2\epsilon_{0}.$$

Now, \begin{align*}
& d\left(\overline{x}_{n(p)+\eta},\overline{x}_{m(p)+\eta}\right)\\
& = d\left( \mathcal{U}\left(\left(\overline{x}_{n(p)+\i-1,\widehat{ n(p)+\eta-1}}\right)_{\i=1}^{\infty}\right), T\left(\left(\overline{x}_{m(p)+\i-1,\widehat{ m(p)+\eta-1}}\right)_{\i=1}^{\infty}\right)\right)\\
& \leq \beta \left(\mathcal{M}'_{\mathcal{U}_\eta}\left(\left(\overline{x}_{n(p)+\i-1,\widehat{ n(p)+\eta-1}}\right)_{\i=1}^{\infty},\left(\overline{x}_{m(p)+\i-1,\widehat{ m(p)+\eta-1}}\right)_{\i=1}^{\infty}\right)\right)\\&  ~~~~~\mathcal{M}'_{\mathcal{U}_\eta}\left( \left(\overline{x}_{n(p)+\i-1,\widehat{ n(p)+\eta-1}}\right)_{\i=1}^{\infty},\left(\overline{x}_{m(p)+\i-1,\widehat{ m(p)+\eta-1}}\right)_{\i=1}^{\infty}\right).
\end{align*}

This implies, \begin{align*}
\epsilon_0\leq & \lim_{p\rightarrow\infty} \beta \left(\mathcal{M}'_{\mathcal{U}_\eta}\left(\left(\overline{x}_{n(p)+\i-1,\widehat{ n(p)+\eta-1}}\right)_{\i=1}^{\infty},\left(\overline{x}_{m(p)+\i-1,\widehat{ m(p)+\eta-1}}\right)_{\i=1}^{\infty}\right)\right)\cdot 2\epsilon_0\\
\text{i.e., } & \lim_{p\rightarrow\infty} \beta \left(\mathcal{M}'_{\mathcal{U}_\eta}\left(\left(\overline{x}_{n(p)+\i-1,\widehat{ n(p)+\eta-1}}\right)_{\i=1}^{\infty},\left(\overline{x}_{m(p)+\i-1,\widehat{ m(p)+\eta-1}}\right)_{\i=1}^{\infty}\right)\right)\geq \frac{1}{2}\\
\text{i.e., }& \lim_{p\rightarrow\infty} \beta \left(M'_{T_\eta}\left( \left(\overline{x}_{n(p)+\i-1,\widehat{ n(p)+\eta-1}}\right)_{\i=1}^{\infty},\left(\overline{x}_{m(p)+\i-1,\widehat{ m(p)+\eta-1}}\right)_{\i=1}^{\infty}\right)\right)=\frac{1}{2}\\
\text{i.e., }&\lim_{p\rightarrow\infty} \mathcal{M}'_{T_\eta}\left( \left(\overline{x}_{n(p)+\i-1,\widehat{ n(p)+\eta-1}}\right)_{\i=1}^{\infty},\left(\overline{x}_{m(p)+\i-1,\widehat{ m(p)+\eta-1}}\right)_{\i=1}^{\infty}\right)=0.
\end{align*}
Therefore, $2\epsilon_0=0~ \Rightarrow \epsilon_0=0.$ Contradiction. Therefore, $\{\overline{x}_{n+\eta}\}_{n\in \mathbb{N}}$ is a Cauchy sequence.

Since, $\overline{X}$ is a complete metric space, there exists $\overline{w}\in \overline{X}$ such that $\lim\limits_{n\rightarrow \infty} \overline{x}_{n+\eta}=\overline{w}. $

Next, we show that $\overline{w}$ is a fixed point of $\mathcal{U}$. We have, \begin{align*}
& \mathcal{M}'_{\mathcal{U}_\eta}\left(\left(\overline{x}_{n+\i-1, \widehat{n+\eta-1}}\right)_{\i=1}^{\infty}, \left(\overline{w}\right)_{\i=1}^{\infty}\right)\\
& = \alpha B_\eta \left(\left(\overline{x}_{n+\i-1, \widehat{n+\eta-1}}\right)_{\i=1}^{\infty}, \left(\overline{w}\right)_{\i=1}^{\infty}\right)+ 2\gamma K_\eta \left(\left(\overline{x}_{n+\i-1, \widehat{n+\eta-1}}\right)_{\i=1}^{\infty}, \left(\overline{w}\right)_{\i=1}^{\infty}\right)+2\delta F_\eta\left(\left(\overline{x}_{n+\i-1, \widehat{n+\eta-1}}\right)_{\i=1}^{\infty}, \left(\overline{w}\right)_{\i=1}^{\infty}\right)
\end{align*}

\noindent Now, \begin{align*}
&B_\eta \left(\left(\overline{x}_{n+\i-1, \widehat{n+\eta-1}}\right)_{\i=1}^{\infty}, \left(\overline{w}\right)_{\i=1}^{\infty}\right)\\
& = d\left(\overline{x}_{n+\eta-1},\overline{w}\right)
\end{align*}

Therefore, $$\lim_{n\rightarrow \infty} B_\eta \left(\left(\overline{x}_{n+\i-1, \widehat{n+\eta-1}}\right)_{\i=1}^{\infty}, \left(\overline{w}\right)_{\i=1}^{\infty}\right)= 0$$
and
\begin{align*}
&K_\eta \left(\left(\overline{x}_{n+\i-1, \widehat{n+\eta-1}}\right)_{\i=1}^{\infty}, \left(\overline{w}\right)_{\i=1}^{\infty}\right)\\
& = d\left(\mathcal{U}\left(\left(\overline{x}_{n+\i-1, \widehat{n+\eta-1}}\right)_{\i=1}^{\infty}\right),\overline{x}_{n+\eta-1}\right)+d\left(\mathcal{U}\left(\left(\overline{w}\right)_{\i=1}^{\infty}\right),\overline{w}\right)\\
& = d\left(\overline{x}_{n+\eta},\overline{x}_{n+\eta-1}\right)+d\left(\mathcal{U}\left(\left(\overline{w}\right)_{\i=1}^{\infty}\right),\overline{w}\right).
\end{align*}

Therefore, $$\lim_{n\rightarrow \infty} K_\eta \left(\left(\overline{x}_{n+\i-1, \widehat{n+\eta-1}}\right)_{\i=1}^{\infty}, \left(\overline{w}\right)_{\i=1}^{\infty}\right)= d\left(\mathcal{U}\left(\left(\overline{w}\right)_{\i=1}^{\infty}\right),\overline{w}\right).$$

 Next, we show that, $$ \lim_{n\rightarrow \infty}  F_\eta\left(\left(\overline{x}_{n+\i-1, \widehat{n+\eta-1}}\right)_{\i=1}^{\infty}, \left(\overline{w}\right)_{\i=1}^{\infty}\right)= d\left(\mathcal{U}\left(\left(\overline{w}\right)_{\i=1}^{\infty}\right),\overline{w}\right).$$

We have,
\begin{align*}
& F_\eta\left(\left(\overline{x}_{n+\i-1, \widehat{n+\eta-1}}\right)_{\i=1}^{\infty}, \left(\overline{w}\right)_{\i=1}^{\infty}\right)\\
& = d\left(\mathcal{U}\left(\left(\overline{x}_{n+\i-1, \widehat{n+\eta-1}}\right)_{\i=1}^{\infty}\right),\overline{w}\right)+d\left(\mathcal{U}\left(\left(\overline{w}\right)_{\i=1}^{\infty}\right),\overline{x}_{n+\eta-1}\right)\\
& \leq d\left(\overline{x}_{n+\eta},\overline{w}\right)+d\left(\mathcal{U}\left(\left(\overline{w}\right)_{\i=1}^{\infty}\right),\overline{w}\right)+d\left(\overline{w},\overline{x}_{n+\eta-1}\right).
\end{align*}

This implies,
$$ \lim_{n\rightarrow \infty}  F_\eta\left(\left(\overline{x}_{n+\i-1, \widehat{n+\eta-1}}\right)_{\i=1}^{\infty}, \left(\overline{w}\right)_{\i=1}^{\infty}\right)\leq d\left(\mathcal{U}\left(\left(\overline{w}\right)_{\i=1}^{\infty}\right),\overline{w}\right).$$
{Also, }$$d\left(\mathcal{U}\left(\left(\overline{w}\right)_{\i=1}^{\infty}\right),\overline{w}\right)\leq d\left(\overline{x}_{n+\eta},\overline{w}\right)+d\left(\overline{x}_{n+\eta},\overline{x}_{n+\eta-1}\right)+d\left(\overline{x}_{n+\eta-1},\mathcal{U}\left(\left(\overline{w}\right)_{\i=1}^{\infty}\right)\right)$$
$$\text{i.e., } \lim_{n\rightarrow \infty}  F_\eta\left(\left(\overline{x}_{n+\i-1, \widehat{n+\eta-1}}\right)_{\i=1}^{\infty}, \left(\overline{w}\right)_{\i=1}^{\infty}\right)\geq d\left(\mathcal{U}\left(\left(\overline{w}\right)_{\i=1}^{\infty}\right),\overline{w}\right)$$
Thus, $$ \lim_{n\rightarrow \infty}  F_\eta\left(\left(\overline{x}_{n+\i-1, \widehat{n+\eta-1}}\right)_{\i=1}^{\infty}, \left(\overline{w}\right)_{\i=1}^{\infty}\right)= d\left(\mathcal{U}\left(\left(\overline{w}\right)_{\i=1}^{\infty}\right),\overline{w}\right).$$

Again,
\begin{align*}
& d\left(\mathcal{U}\left(\left(\overline{w}\right)_{\i=1}^{\infty}\right),\overline{w}\right) \leq d\left(\overline{x}_{n+\eta},\overline{w}\right)+d\left(\mathcal{U}\left(\left(\overline{x}_{n+\i-1, \widehat{n+\eta-1}}\right)_{\i=1}^{\infty}\right),\mathcal{U}\left(\left(\overline{w}\right)_{\i=1}^{\infty}\right)\right)\\
& \leq d\left(\overline{x}_{n+\eta},\overline{w}\right)+\beta\left(\mathcal{M}'_{\mathcal{U}_\eta}\left(\left(\overline{x}_{n+\i-1, \widehat{n+\eta-1}}\right)_{\i=1}^{\infty}, \left(\overline{w}\right)_{\i=1}^{\infty}\right)\right)\mathcal{M}'_{\mathcal{U}_\eta}\left(\left(\overline{x}_{n+\i-1, \widehat{n+\eta-1}}\right)_{\i=1}^{\infty}, \left(\overline{w}\right)_{\i=1}^{\infty}\right)\\
\text{Therefore, }&  d\left(\mathcal{U}\left(\left(\overline{w}\right)_{\i=1}^{\infty}\right),\overline{w}\right)\\
& \leq \lim_{n\rightarrow \infty}\beta\left(\mathcal{M}'_{\mathcal{U}_\eta}\left(\left(\overline{x}_{n+\i-1, \widehat{n+\eta-1}}\right)_{\i=1}^{\infty}, \left(\overline{w}\right)_{\i=1}^{\infty}\right)\right)d\left(\mathcal{U}\left(\left(\overline{w}\right)_{\i=1}^{\infty}\right),\overline{w}\right).
\end{align*}

Then \begin{align*}
&  \lim_{n\rightarrow \infty}\beta\left(\mathcal{M}'_{\mathcal{U}_\eta}\left(\left(\overline{x}_{n+\i-1, \widehat{n+\eta-1}}\right)_{\i=1}^{\infty}, \left(\overline{w}\right)_{\i=1}^{\infty}\right)\right)\geq 1\geq \frac{1}{2}\\
\Rightarrow & \lim_{n\rightarrow \infty}\beta\left(\mathcal{M}'_{\mathcal{U}_\eta}\left(\left(\overline{x}_{n+\i-1, \widehat{n+\eta-1}}\right)_{\i=1}^{\infty}, \left(\overline{w}\right)_{\i=1}^{\infty}\right)\right)=\frac{1}{2}\\
\Rightarrow & \lim_{n\rightarrow \infty} \mathcal{M}'_{\mathcal{U}_\eta}\left(\left(\overline{x}_{n+\i-1, \widehat{n+\eta-1}}\right)_{\i=1}^{\infty}, \left(\overline{w}\right)_{\i=1}^{\infty}\right)=0\\
\Rightarrow & d\left(\mathcal{U}\left(\left(\overline{w}\right)_{\i=1}^{\infty}\right),\overline{w}\right)=0; \text{ i.e., } \mathcal{U}\left(\left(\overline{w}\right)_{\i=1}^{\infty}\right)=\overline{w}.\\
\end{align*}

Therefore, $\mathcal{U}$ has a fixed point.

Next, we prove that the fixed point of $\mathcal{U}$ is unique. Let us suppose, on the contrary, that $T$ has two different fixed points $\overline{w}$ and $\overline{v}.$ Then, $\mathcal{U}\left(\left(\overline{w}\right)_{i=1}^{\infty}\right)=\overline{w}$ and $\mathcal{U}\left(\left(\overline{v}\right)_{i=1}^{\infty}\right)=\overline{v}.$

In this conditions we have,\begin{align*}
&d(\overline{w},\overline{v})=d\left(\mathcal{U}\left(\left(\overline{w}\right)_{\i=1}^{\infty}\right),\mathcal{U}\left(\left(\overline{v}\right)_{\i=1}^{\infty}\right)\right)\\
& \leq \beta\left(\mathcal{M}'_{\mathcal{U}_\eta}\left(\left(\overline{w}\right)_{\i=1}^{\infty}, \left(\overline{v}\right)_{\i=1}^{\infty}\right)\right)\mathcal{M}'_{\mathcal{U}_\eta}\left(\left(\overline{w}\right)_{\i=1}^{\infty}, \left(\overline{v}\right)_{\i=1}^{\infty}\right)\\
& < \frac{1}{2}\left[\alpha \{d\left(\overline{w},\overline{v}\right)+ 2\gamma\{d\left(\mathcal{U}\left(\left(\overline{w}\right)_{\i=1}^{\infty}\right),\overline{w}\right)+d\left(\mathcal{U}\left(\left(\overline{v}\right)_{\i=1}^{\infty}\right),\overline{v}\right)\}\right.\\
& \left. ~~~~+2\delta \{d\left(\mathcal{U}\left(\left(\overline{w}\right)_{\i=1}^{\infty}\right),\overline{v}\right)+d\left(\mathcal{U}\left(\left(\overline{v}\right)_{\i=1}^{\infty}\right),\overline{w}\right)\}\right]\\
& = \frac{1}{2}\left[\alpha d\left(\overline{w},\overline{v}\right)+2\gamma \{d\left(\overline{w},\overline{w}\right)+d\left(\overline{v},\overline{v}\right)\}+2\delta\{d\left(\overline{w},\overline{v}\right)+d\left(\overline{v},\overline{w}\right)\}\right]\\
& <\frac{1}{2} \{d\left(\overline{w},\overline{v}\right)+d\left(\overline{v},\overline{w}\right)\}\\
&= d\left(\overline{w},\overline{v}\right)\\
\text{i.e., }d(\overline{w},\overline{v})<d(\overline{w},\overline{v}).
\end{align*}
Contradiction. Therefore, $w=v.$ This completes the proof.
\end{proof}

\begin{example}
Let $\overline{X}=[0,1]$ be a metric space with usual metric. Let us define $\mathcal{U}:\prod\limits_{\i=1}^{\infty}\overline{X}\rightarrow \overline{X}$ by:
\[\mathcal{U}((\overline{x}_{\i,\widehat{\eta}})_{\i=1}^{\infty})=\begin{cases}
\frac{\overline{x}_\eta^2}{30} & ~ \text{if}~ \overline{x}_\eta\in [0,1);\\
\frac{1}{60} & ~ \text{if}~ \overline{x}_\eta =1.
\end{cases}
\]
Then, $\mathcal{U}$ is a generalized $\mathcal{C}_\eta$ contraction for $\beta =\frac{3}{8}$, $\alpha =\frac{1}{2}$ $\gamma=\frac{1}{8}$ and $\delta=\frac{1}{8}$. Therefore, $\mathcal{U}$ has a unique fixed point $(0,0,...).$
\end{example}

\begin{corollary}\label{BKF}
Let $\overline{X}$ be a complete metric space and $\mathcal{U}:\prod\limits_{\i=1}^{\infty}\overline{X}\rightarrow \overline{X}$ be a Banach's (Kannan's or Fisher's) generalized contraction mapping for infinite dimension. Then $\mathcal{U}$ has a unique fixed point.
\end{corollary}

\begin{proof}
If we choose $\alpha =1$ ($\gamma=\frac{1}{2}$ or $\delta =\frac{1}{2}$ respectively) in Theorem \ref{3} the conclusion follow.
\end{proof}

\begin{theorem}\label{9}
Let $\left(\overline{X},d\right)$ be a complete metric space and let $\mathcal{U}:\prod\limits_{l=1}^\infty \overline{X} \rightarrow \overline{X}$ be a generalized $L^{\alpha}_\eta$ $contraction$ for some $ \eta\in\mathbb{N}.$ Then the following holds: \begin{itemize}
\item[$(1)$] $\mathcal{U}$ has a unique fixed point, i.e. there exists $\overline{w}\in \overline{X}$ such that $\mathcal{U}((\overline{w})_{\i=1}^{\infty})=\overline{w}.$
\item[$(2)$] For any $\overline{x}_1, \overline{x}_2, ...,\overline{x}_\eta\in X,$ the infinite $\eta-$Picard sequence converges to $\overline{w}.$
\end{itemize}
\end{theorem}

\begin{proof}
\textit{Case I}: Let $f=f_1\in \mathcal{F}_1$ in Definition \ref{LC}; then, using the proof of Theorem \ref{3}, we get the desired conclusion.\\
\textit{Case II}: Let $f=f_2 \in \mathcal{F}_2$, or $f=f_3 \in \mathcal{F}_3$ in Definition \ref{LC}. In each case we get one of the special case of Banach's, Kannan's or Fisher's generalized contraction mapping for infinite dimension. Then, from Corollary \ref{BKF}, we get the complete proof of the result.
\end{proof}

\begin{theorem}\label{5}
Let $\left(\overline{X},d\right)$ be a complete metric space and let $\mathcal{U}:\prod\limits_{l=1}^{\eta} \overline{X} \rightarrow \overline{X}$ be a $\mathcal{C}_{\eta}$ contraction for some $ \eta\in\mathbb{N}.$ Then $\mathcal{U}$ has a unique fixed point.
\end{theorem}
\begin{proof}

Taking the finite Picard sequence $\overline{x}_{n+\eta}=\mathcal{U}((\overline{x}_{n+\i-1})_{\i=1}^{\eta})$ in Theorem \ref{3} we get the conclusion.
\end{proof}

\hspace{-.5cm}
\begin{theorem}\label{4}
Let $\left(\overline{X},d\right)$ be a complete metric space and let $\mathcal{U}:\prod\limits_{l=1}^\infty \overline{X} \rightarrow \overline{X}$ be a generalized $H^{\alpha}_{\eta}$ $contraction$ for some $ \eta\in\mathbb{N}.$ Then $\mathcal{U}$ has a unique fixed point.
\end{theorem}
\begin{proof}
Taking $\alpha =0, \gamma=\frac{\alpha}{2} \text{ and } \delta =\frac{(1-\alpha )}{2}$ respectively) in Theorem \ref{3} we get the conclusion.
\end{proof}
\hspace{-.5cm}
\begin{example}
Let $\overline{X}=[0,1]$ be a metric space with usual metric. Let us define $\mathcal{U}:\prod\limits_{i=1}^{\infty}\overline{X}\rightarrow \overline{X}$ by:
\[\mathcal{U}((\overline{x}_{\i,\widehat{\eta}})_{\i=1}^{\infty})=\begin{cases}
\frac{\overline{x}_\eta}{10} & ~ \text{if}~ \overline{x}_\eta\in [0,1);\\
\frac{1}{25} & ~ \text{if}~ \overline{x}_\eta =1.
\end{cases}
\]
Then $\mathcal{U}$ is a generalized $H^{\alpha}_\eta$ contraction, for $\beta =\frac{1}{4}$ and $\alpha =\frac{1}{2}.$ Therefore, $\mathcal{U}$ has a unique fixed point $(0,0,...).$
\end{example}
\hspace{8pt}
\begin{theorem}
Let $\left(\overline{X},d\right)$ be a complete metric space and let $\mathcal{U}:\prod\limits_{l=1}^{\eta} \overline{X} \rightarrow \overline{X}$ be a $H_{\eta}^{\alpha}$ contraction for some $ \eta\in\mathbb{N}.$ Then $\mathcal{U}$ has a unique fixed point.
\end{theorem}
\begin{proof}
Taking the finite Picard sequence $\overline{x}_{n+\eta}=\mathcal{U}((\overline{x}_{n+\i-1})_{\i=1}^{\eta})$ in Theorem \ref{4} we get the conclusion.
\end{proof}
\hspace{-.5cm}

\begin{corollary}
Let $X$ be a complete metric space and $\mathcal{U}:\prod\limits_{\i=1}^{\eta}\overline{X}\rightarrow \overline{X}$ be a Banach's (Kannan's or Fisher's) generalized contraction mapping for finite dimensions. Then $\mathcal{U}$ has a unique fixed point.
\end{corollary}

\begin{proof}
 Taking $\alpha =1$ ($\gamma=\frac{1}{2}$ or $\delta =\frac{1}{2}$ respectively) in Theorem \ref{5} the conclusion follow.
\end{proof}

\section{Application to Volterra type integral equations}

\quad The Volterra integral equation was introduced by Vito Volterra in the end of $19^{th}$ century. Later, at beginning of the $20^{th}$ century, the Romanian mathematician Traian Lalescu studied in his PhD thesis this type of equations, under the supervision of \'Emile Picard. Volterra integral equations have many applications in different domains as demography, economic processes, physics and astrophysics. Then, the aim of our section is to put in evidence the fixed point method for proving the existence of a solution for a Volterra type integral equation.

\quad Let us recall the usual Volterra integral equation of first kind as
$$\overline{w}(t)= \lambda \int_{a}^{t} \mathbb{K}(t,\overline{w}(s))ds,$$
where $\lambda$ is a non-zero real and $\mathbb{K}:[a,b]\times \mathbb{R}\rightarrow \mathbb{R}$ is a real valued continuous function. We generalise in this way the notion of usual Volterra integral equation.

We know that $\mathbb{R}^{\omega}$ is the countable product of $\mathbb{R}$. We restricted this set and define $\mathbb{R}^{\omega}_\eta$ as, $$\mathbb{R}^{\omega}_\eta=\{\overline{w}: \overline{w}=(\overline{w}_1,\overline{w}_2,...,\overline{w}_\eta,\overline{w}_\eta,... )\};$$  for some $\eta\in\mathbb{N}.$ Let $[a,b]\subset\mathbb{R},$ then $D=[a,b]\times[a,b]\times[a,b]\times...$ is a subset of $\mathbb{R}^{\omega}$.

 Let $\mathcal{C}[a,b]$ denote the set of all continuous function on $[a,b].$ Then, $d: \mathcal{C}[a,b]\times\mathcal{C}[a,b]\rightarrow \mathbb{R}_{+}$ be defined by \begin{align*}
 d\left(f(\overline{x}),g(\overline{x})\right)=\sup_{\overline{x}\in [a,b]}\left|f(\overline{x})-g(\overline{x})\right|e^{-m\overline{x}} \tag{5}\label{eq_tag5}
 \end{align*} for all $f(\overline{x}),~g(\overline{x})\in\mathcal{C}[a,b].$ So, $(\mathcal{C}[a,b],d)$ is a complete metric space.

Let us define a integral equation, where the unknown function is an infinite-tuple of the special form (element of $\mathbb{R}^{\omega}_\eta$) as follows
\begin{align*}
(\overline{w}_{i,\hat{\eta}})_{\i=1}^{\infty}(t)=\lambda \int_{a}^{t} \mathbb{K}(t,(\overline{w}_{\i,\hat{\eta}})_{\i=1}^{\infty}(s))ds; \tag{6}\label{eq_tag6}
\end{align*}
where $\mathbb{K}:[a,b]\times \mathbb{R}^{\omega}_\eta\rightarrow \mathbb{R}$ is a continuous function.

In the section, we assured the existence and uniqueness of the solution of the integral equation \eqref{eq_tag6}.
\begin{theorem}
Let $\mathbb{K}:[a,b]\times \mathbb{R}^{\omega}_\eta\rightarrow \mathbb{R}$ be a continuous function and $f\in \mathbb{G}_{\frac{1}{2}}$. If the given condition holds:
\begin{align*}
& \left|\mathbb{K}\left(\overline{x},(\overline{w}_{\i,\hat{\eta}})_{\i=1}^{\infty}(t)\right)-\mathbb{K}\left(\overline{x},(\overline{v}_{\i,\hat{\eta}})_{\i=1}^{\infty}(t)\right)\right|\\
& \leq f\left(L\left(\overline{x},(\overline{w}_{\i,\hat{\eta}})_{\i=1}^{\infty}(t),(\overline{v}_{\i,\hat{\eta}})_{\i=1}^{\infty}(t))\right)e^{-mt}\right)L\left(\overline{x},(\overline{w}_{\i,\hat{\eta}})_{\i=1}^{\infty}(t),(\overline{v}_{\i,\hat{\eta}})_{\i=1}^{\infty}(t))\right)\\
\end{align*}
where,
\begin{align*}
&L\left(\overline{x},(\overline{w}_{\i,\hat{\eta}})_{\i=1}^{\infty}(t),(\overline{v}_{\i,\hat{\eta}})_{\i=1}^{\infty}(t))\right)\\
&=\alpha\left\{\left|\overline{w}_\eta(t)-\lambda\int_{a}^{\overline{x}} \mathbb{K}(\overline{x},(\overline{w}_{\i,\hat{\eta}})_{\i=1}^{\infty}(t))dt\right|+\left|\overline{v}_\eta(t)-\lambda\int_{a}^{\overline{x}} \mathbb{K}(\overline{x},(\overline{v}_{\i,\hat{\eta}})_{\i=1}^{\infty}(t))dt\right|\right\}+\\
&(1-\alpha)\left\{\left|\overline{w}_\eta(t)-\lambda\int_{a}^{\overline{x}} \mathbb{K}(\overline{x},(\overline{v}_{\i,\hat{\eta}})_{\i=1}^{\infty}(t))dt\right|+\left|\overline{v}_\eta(t)-\lambda\int_{a}^{\overline{x}} \mathbb{K}(\overline{x},(\overline{w}_{\i,\hat{\eta}})_{\i=1}^{\infty}(t))dt\right|\right\};
\end{align*}
for some $\alpha \in [0,1];$
then, the integral equation \eqref{eq_tag6} has a unique solution.
\end{theorem}

\begin{proof}
In order to find a solution of the integral equation \eqref{eq_tag6} we have to transform the problem into an equivalent fixed point problem.\\ Let us define a function $T:\mathcal{C}_{\eta}(D)\rightarrow\mathcal{C}[a,b]$ by, $$\mathcal{U}\left((\overline{w}_{\i,\hat{\eta}})_{\i=1}^{\infty}(\overline{x})\right)=\lambda \int_{a}^{\overline{x}} \mathbb{K}\left(\overline{x},(\overline{w}_{\i,\hat{\eta}})_{\i=1}^{\infty}(t)\right)dt;$$ where $\mathcal{C}_{\eta}(D)$ is considered as a subset of $\mathbb{R}^{\omega}_\eta.$\\ Therefore, we have to find the fixed point of $\mathcal{U}$.\\ Now, using \eqref{eq_tag5}, we have, \begin{align*}
&d\left(\mathcal{U}\left((\overline{w}_{\i,\hat{\eta}})_{\i=1}^{\infty}(\overline{x})\right),\mathcal{U}\left((\overline{v}_{\i,\hat{\eta}})_{\i=1}^{\infty}(\overline{x})\right)\right)\\
& = \sup_{\overline{x}\in [a,b]}\left|\mathcal{U}\left((\overline{w}_{\i,\hat{\eta}})_{\i=1}^{\infty}(\overline{x})\right)-\mathcal{U}\left((\overline{v}_{\i,\hat{\eta}})_{\i=1}^{\infty}(\overline{x})\right)\right|e^{-m\overline{x}}\\
&\leq \sup_{\overline{x}\in [a,b]}\left|\lambda \int_{a}^{\overline{x}} \mathbb{K}\left(\overline{x},(\overline{w}_{\i,\hat{\eta}})_{\i=1}^{\infty}(t)\right)dt-\lambda \int_{a}^{\overline{x}} \mathbb{K}\left(\overline{x},(\overline{v}_{\i,\hat{\eta}})_{\i=1}^{\infty}(t)\right)dt\right|e^{-m\overline{x}}\\
&\leq|\lambda| \sup_{\overline{x}\in [a,b]}\left|\int_{a}^{\overline{x}} \mathbb{K}\left(\overline{x},(\overline{w}_{\i,\hat{\eta}})_{\i=1}^{\infty}(t)\right)dt- \int_{a}^{\overline{x}} \mathbb{K}\left(\overline{x},(\overline{v}_{\i,\hat{\eta}})_{\i=1}^{\infty}(t)\right)dt\right|e^{-m\overline{x}}\\
&\leq|\lambda| \sup_{\overline{x}\in [a,b]}\int_{a}^{\overline{x}} \left|\mathbb{K}\left(\overline{x},(\overline{w}_{\i,\hat{\eta}})_{\i=1}^{\infty}(t)\right)- K\left(\overline{x},(\overline{v}_{\i,\hat{\eta}})_{\i=1}^{\infty}(t)\right)\right|dt~e^{-m\overline{x}}\\
&\leq|\lambda| \sup_{\overline{x}\in [a,b]}\int_{a}^{\overline{x}} \leq f\left(L\left(\overline{x},(\overline{w}_{\i,\hat{\eta}})_{\i=1}^{\infty}(t),(\overline{v}_{i,\hat{\eta}})_{\i=1}^{\infty}(t))\right)e^{-mt}\right)L\left(\overline{x},(\overline{w}_{\i,\hat{\eta}})_{\i=1}^{\infty}(t),(\overline{v}_{\i,\hat{\eta}})_{\i=1}^{\infty}(t))\right)e^{-m\overline{x}}.\\
\end{align*}

We have,
\begin{align*}
\begin{split}
&L\left(\overline{x},(\overline{w}_{\i,\hat{\eta}})_{\i=1}^{\infty}(t),(\overline{v}_{\i,\hat{\eta}})_{\i=1}^{\infty}(t))\right)e^{-m\overline{x}} \\
=&\left[\alpha\left\{\left|\overline{w}_\eta(t)-\lambda\int_{a}^{\overline{x}} \mathbb{K}(\overline{x},(\overline{w}_{\i,\hat{\eta}})_{\i=1}^{\infty}(s))ds\right|+\left|\overline{v}_\eta(t)-\lambda\int_{a}^{\overline{x}} \mathbb{K}(\overline{x},(\overline{v}_{\i,\hat{\eta}})_{\i=1}^{\infty}(s))ds\right|\right\}+ \right.  \\
&\left.(1-\alpha)\left\{\left|\overline{w}_\eta(t)-\lambda\int_{a}^{\overline{x}} \mathbb{K}(\overline{x},(\overline{v}_{\i,\hat{\eta}})_{\i=1}^{\infty}(s))ds\right|+\left|\overline{v}_\eta(t)-\lambda\int_{a}^{\overline{x}} \mathbb{K}(\overline{x},(\overline{w}_{\i,\hat{\eta}})_{\i=1}^{\infty}(s))ds\right|\right\}\right]e^{-m\overline{x}} \\
\leq& \sup_{t\in [a,b]}\left[\alpha\left\{\left|\overline{w}_\eta(t)-\lambda\int_{a}^{\overline{x}} \mathbb{K}(\overline{x},(\overline{w}_{\i,\hat{\eta}})_{\i=1}^{\infty}(s))ds\right|+\left|\overline{v}_\eta(t)-\lambda\int_{a}^{\overline{x}} \mathbb{K}(\overline{x},(\overline{v}_{\i,\hat{\eta}})_{\i=1}^{\infty}(s))ds\right|\right\}+ \right.  \\
&\left.(1-\alpha)\left\{\left|\overline{w}_\eta(t)-\lambda\int_{a}^{\overline{x}} \mathbb{K}(\overline{x},(\overline{v}_{\i,\hat{\eta}})_{\i=1}^{\infty}(s))ds\right|+\left|\overline{v}_\eta(t)-\lambda\int_{a}^{\overline{x}} \mathbb{K}(\overline{x},(\overline{w}_{\i,\hat{\eta}})_{\i=1}^{\infty}(s))ds\right|\right\}\right]e^{-m\overline{x}} \\
\leq& \sup_{t\in [a,b]}\left[\alpha\left\{\left|\overline{w}_\eta(t)-\mathcal{U}\left((\overline{w}_{\i,\hat{\eta}})_{\i=1}^{\infty}(\overline{x})\right)\right|+\left|\overline{v}_\eta(t)-\mathcal{U}\left((\overline{v}_{\i,\hat{\eta}})_{\i=1}^{\infty}(\overline{x})\right)\right|\right\}+ \right.  \\
&\left.(1-\alpha)\left\{\left|\overline{w}_\eta(t)-\mathcal{U}\left((\overline{v}_{\,\hat{\eta}})_{\i=1}^{\infty}(\overline{x})\right)\right|+\left|\overline{v}_\eta(t)-\mathcal{U}\left((\overline{w}_{\i,\hat{\eta}})_{\i=1}^{\infty}(\overline{x})\right)\right|\right\}\right]e^{-mt}e^{-m(\overline{x}-t)} \\
\leq& ~\alpha\left\{d\left(\overline{w}_\eta(x),\mathcal{U}\left((\overline{w}_{\i,\hat{\eta}})_{\i=1}^{\infty}(\overline{x})\right)\right)+d\left(\overline{v}_\eta(x),\mathcal{U}\left((\overline{v}_{\i,\hat{\eta}})_{\i=1}^{\infty}(\overline{x})\right)\right)\right\}+   \\
&(1-\alpha)\left\{d\left(\overline{w}_\eta(\overline{x}),\mathcal{U}\left((\overline{v}_{\i,\hat{\eta}})_{\i=1}^{\infty}(\overline{x})\right)\right)+d\left(\overline{v}_\eta(\overline{x}),\mathcal{U}\left((\overline{w}_{\i,\hat{\eta}})_{\i=1}^{\infty}(\overline{x})\right)\right)\right\}e^{-m(\overline{x}-t)} . \\
\end{split}
\end{align*}

Therefore,\begin{align*}
\begin{split}
&f\left(L\left(\overline{x},(\overline{w}_{\i,\hat{\eta}})_{i=1}^{\infty}(t),(\overline{v}_{\i,\hat{\eta}})_{i=1}^{\infty}(t))\right)e^{-mt}\right) \\
=& f\left(\alpha\left\{d\left(\overline{w}_\eta(t),\mathcal{U}\left((\overline{w}_{\i,\hat{\eta}})_{\i=1}^{\infty}(t)\right)\right)+d\left(\overline{v}_\eta(t),\mathcal{U}\left((\overline{v}_{\i,\hat{\eta}})_{\i=1}^{\infty}(t)\right)\right)\right\}+ \right.  \\
&\left. (1-\alpha)\left\{d\left(\overline{w}_\eta(t),\mathcal{U}\left((\overline{v}_{\i,\hat{\eta}})_{\i=1}^{\infty}(t)\right)\right)+d\left(\overline{v}_\eta(t),\mathcal{U}\left((\overline{w}_{\i,\hat{\eta}})_{\i=1}^{\infty}(t)\right)\right)\right\}\right). \\
\end{split}
\end{align*}
Thus, we get,\begin{align*}
\begin{split}
&d\left(\mathcal{U}\left((\overline{w}_{i,\hat{\eta}})_{\i=1}^{\infty}(\overline{x})\right),\mathcal{U}\left((\overline{v}_{\i,\hat{\eta}})_{\i=1}^{\infty}(\overline{x})\right)\right)\\
&\leq|\lambda| \sup_{\overline{x}\in [a,b]} \int_{a}^{\overline{x}}f\left(\alpha\left\{d\left(\overline{w}_\eta(\overline{x}),\mathcal{U}\left((\overline{w}_{\i,\hat{\eta}})_{\i=1}^{\infty}(\overline{x})\right)\right)+d\left(\overline{v}_\eta(x),\mathcal{U}\left((\overline{v}_{\i,\hat{\eta}})_{\i=1}^{\infty}(\overline{x})\right)\right)\right\}+ \right.  \\
&\left.(1-\alpha)\left\{d\left(\overline{w}_\eta(\overline{x}),\mathcal{U}\left((\overline{v}_{\i,\hat{\eta}})_{\i=1}^{\infty}(\overline{x})\right)\right)+d\left(\overline{v}_\eta(\overline{x}),\mathcal{U}\left((\overline{w}_{\i,\hat{\eta}})_{\i=1}^{\infty}(\overline{x})\right)\right)\right\} \right)\\
&\left[\alpha\left\{d\left(\overline{w}_\eta(\overline{x}),\mathcal{U}\left((\overline{w}_{\i,\hat{\eta}})_{\i=1}^{\infty}(\overline{x})\right)\right)+d\left(\overline{v}_\eta(\overline{x}),\mathcal{U}\left((\overline{v}_{\i,\hat{\eta}})_{\i=1}^{\infty}(\overline{x})\right)\right)\right\}+ \right.  \\
&\left.(1-\alpha)\left\{d\left(\overline{w}_\eta(\overline{x}),\mathcal{U}\left((\overline{v}_{\i,\hat{\eta}})_{\i=1}^{\infty}(\overline{x})\right)\right)+d\left(\overline{v}_\eta(\overline{x}),\mathcal{U}\left((\overline{w}_{\i,\hat{\eta}})_{\i=1}^{\infty}(\overline{x})\right)\right)\right\}\right]e^{-m(\overline{x}-t)}dt.
\end{split}
\end{align*}
Therefore,\begin{align*}
\begin{split}
&d\left(T\left((\overline{w}_{\i,\hat{\eta}})_{\i=1}^{\infty}(\overline{x})\right),\mathcal{U}\left((\overline{v}_{\i,\hat{\eta}})_{\i=1}^{\infty}(\overline{x})\right)\right)\\
&\leq \frac{|\lambda| }{m}f\left(\alpha\left\{d\left(\overline{w}_\eta(\overline{x}),\mathcal{U}\left((\overline{w}_{\i,\hat{\eta}})_{\i=1}^{\infty}(\overline{x})\right)\right)+d\left(\overline{v}_\eta(\overline{x}),\mathcal{U}\left((\overline{v}_{\i,\hat{\eta}})_{\i=1}^{\infty}(\overline{x})\right)\right)\right\}+ \right.  \\
&\left.(1-\alpha)\left\{d\left(\overline{w}_\eta(\overline{x}),\mathcal{U}\left((\overline{v}_{\i,\hat{\eta}})_{\i=1}^{\infty}(\overline{x})\right)\right)+d\left(\overline{v}_\eta(\overline{x}),\mathcal{U}\left((\overline{w}_{\i,\hat{\eta}})_{\i=1}^{\infty}(\overline{x})\right)\right)\right\} \right)\\
&\left[\alpha\left\{d\left(\overline{w}_\eta(\overline{x}),\mathcal{U}\left((\overline{w}_{\i,\hat{\eta}})_{\i=1}^{\infty}(\overline{x})\right)\right)+d\left(\overline{v}_\eta(\overline{x}),\mathcal{U}\left((\overline{v}_{\i,\hat{\eta}})_{\i=1}^{\infty}(\overline{x})\right)\right)\right\}+ \right.  \\
&\left.(1-\alpha)\left\{d\left(\overline{w}_\eta(\overline{x}),\mathcal{U}\left((\overline{v}_{\i,\hat{\eta}})_{\i=1}^{\infty}(\overline{x})\right)\right)+d\left(\overline{v}_\eta(\overline{x}),\mathcal{U}\left((\overline{w}_{\i,\hat{\eta}})_{\i=1}^{\infty}(\overline{x})\right)\right)\right\}\right][1-e^{-m(b-a)}]
\end{split}
\end{align*}

If we choose $m\geq|\lambda|~[1-e^{-m(b-a)}]$, we get,
\begin{align*}
\begin{split}
&d\left(\mathcal{U}\left((\overline{w}_{\i,\hat{\eta}})_{\i=1}^{\infty}(\overline{x})\right),\mathcal{U}\left((\overline{v}_{\i,\hat{\eta}})_{\i=1}^{\infty}(\overline{x})\right)\right)\\
\leq& f\left(\alpha\left\{d\left(\overline{w}_\eta(\overline{x}),\mathcal{U}\left((\overline{w}_{\i,\hat{\eta}})_{\i=1}^{\infty}(\overline{x})\right)\right)+d\left(\overline{v}_\eta(\overline{x}),\mathcal{U}\left((\overline{v}_{\i,\hat{\eta}})_{\i=1}^{\infty}(\overline{x})\right)\right)\right\}+ \right.  \\
&\left.(1-\alpha)\left\{d\left(\overline{w}_\eta(\overline{x}),\mathcal{U}\left((\overline{v}_{\i,\hat{\eta}})_{\i=1}^{\infty}(\overline{x})\right)\right)+d\left(\overline{v}_\eta(\overline{x}),\mathcal{U}\left((\overline{w}_{\i,\hat{\overline{\eta}}})_{\i=1}^{\infty}(\overline{x})\right)\right)\right\} \right)\\
&\left[\alpha\left\{d\left(\overline{w}_\eta(\overline{x}),\mathcal{U}\left((\overline{w}_{\i,\hat{\eta}})_{\i=1}^{\infty}(\overline{x})\right)\right)+d\left(\overline{v}_\eta(\overline{x}),\mathcal{U}\left((\overline{v}_{\i,\hat{\eta}})_{\i=1}^{\infty}(\overline{x})\right)\right)\right\}+ \right.  \\
&\left.(1-\alpha)\left\{d\left(\overline{w}_\eta(\overline{x}),\mathcal{U}\left((\overline{v}_{\i,\hat{\eta}})_{\i=1}^{\infty}(\overline{x})\right)\right)+d\left(\overline{v}_\eta(\overline{x}),\mathcal{U}\left((\overline{w}_{\i,\hat{\eta}})_{\i=1}^{\infty}(\overline{x})\right)\right)\right\}\right]\\
\end{split}
\end{align*}

Therefore, $\mathcal{U}$ is a generalized $H_{\alpha}^{\eta}$-contraction mapping and hence, it has a unique fixed point. Then, the integral equation \eqref{eq_tag5} has a unique solution.
\end{proof}

\section{Conclusions}
\quad In this paper we considered some metric hypothesis can be applied to extend the Banach, Kannan and Fisher's theorem in its most general setup possible.  We used this new idea to find some infinite extensions of the other type of contractions and its related operators to prove the existence of fixed points. There are some open questions which arise from these extensions.
\begin{enumerate}
	\item If we consider uncountable product of the complete metric space $(X,d)$ the resultant space $\Pi_{\lambda\in\Lambda}X_{\lambda}$ is not a metric space. Can we prove a generalizations of the above results in these type of topological spaces, which can be expressed as uncountable product of complete metric space?
	\item Can all the results be proved for the notion of coincidence points?
\end{enumerate}

\textbf{Conflicts of Interest} Authors declare that there is no conflicts of interest regarding the publication of this paper.

\textbf{Acknowledgements} The first thank to Suprokash Hazra for his valuable comments in order to improve the paper.

\end{document}